\documentclass{amsart}
\usepackage{amsfonts}

\newtheorem{theorem}{Theorem}
\theoremstyle{plain}

\newtheorem{corollary}{Corollary}

\newtheorem{definition}{Definition}

\newtheorem{lemma}{Lemma}

\newtheorem{proposition}{Proposition}
\newtheorem{remark}{Remark}

\numberwithin{equation}{section}
\input{tcilatex}

\begin{document}
\title[Inequalities for $h-$convex Functions]{New Estimations for $h-$convex
Functions via Further Properties}
\author{Mevl\"{u}t TUN\c{C}$^{\clubsuit }$}
\address{$^{\clubsuit }$Kilis 7 Aral\i k University, Faculty of Science and
Arts, Department of Mathematics, 79000, Kilis, TURKEY}
\email{mevluttunc@kilis.edu.tr}
\author{H\"{u}seyin YILDIRIM$^{\spadesuit }$}
\address{$^{\spadesuit }$Kahramanmara\c{s} S\"{u}t\c{c}\"{u} \.{I}mam
University, Faculty of Science and Arts, Department of Mathematics, 46000,
Kahramanmara\c{s}, TURKEY}
\email{hyildir@ksu.edu.tr}
\keywords{$h-$convex function, Hermite-Hadamard inequality, H\"{o}lder
inequality, Power-mean inequality, Special means.\\
$^{\clubsuit }$Corresponding Author.}

\begin{abstract}
In this paper, some new inequalities of the Hermite-Hadamard type for $h-$
convex functions whose modulus of the derivatives are h-convex and
applications for special means are given.
\end{abstract}

\maketitle

\section{Introduction}

The following definition is well known in the literature [\ref{mit2}]: A
function $f:I\rightarrow 
\mathbb{R}
,$ $\emptyset \neq I\subseteq 
\mathbb{R}
,$ is said to be convex on $I$ if inequality

\begin{equation}
f\left( tx+\left( 1-t\right) y\right) \leq tf\left( x\right) +\left(
1-t\right) f\left( y\right)  \label{101}
\end{equation}%
holds for all $x,y\in I$ and $t\in \left[ 0,1\right] $. Geometrically, this
means that if $P,Q$ and $R$ are three distinct points on the graph of $f$
with $Q$ between $P$ and $R$, then $Q$ is on or below chord $PR$.

Let $f:I\subseteq 
\mathbb{R}
\rightarrow 
\mathbb{R}
$ be a convex function and $a,b\in I$ with $a<b$. The following double
inequality:$\ $

\begin{equation}
f\left( \frac{a+b}{2}\right) \leq \frac{1}{b-a}\int_{a}^{b}f\left( x\right)
dx\leq \frac{f\left( a\right) +f\left( b\right) }{2}  \label{102}
\end{equation}%
is known in the literature as Hadamard's inequality (or H-H inequality) for
convex function. Keep in mind that some of the classical inequalities for
means can come from (\ref{102}) for convenient particular selections of the
function $f$. If $f$ is concave, this double inequality hold in the inversed
way.\bigskip

\begin{definition}
\textit{[\ref{god}] We say that }$f:I\rightarrow 
\mathbb{R}
$\textit{\ is Godunova-Levin function or that }$f$\textit{\ belongs to the
class }$Q\left( I\right) $\textit{\ if }$f$\textit{\ is non-negative and for
all }$x,y\in I$\textit{\ and }$t\in \left( 0,1\right) $\textit{\ we have \ \
\ \ \ \ \ \ \ \ \ \ \ }%
\begin{equation}
f\left( tx+\left( 1-t\right) y\right) \leq \frac{f\left( x\right) }{t}+\frac{%
f\left( y\right) }{1-t}.  \label{103}
\end{equation}
\end{definition}

\begin{definition}
\textit{[\ref{dr1}] We say that }$f:I\subseteq 
\mathbb{R}
\rightarrow 
\mathbb{R}
$\textit{\ is a }$P-$\textit{function or that }$f$\textit{\ belongs to the
class }$P\left( I\right) $\textit{\ if }$f$\textit{\ is nonnegative and for
all }$x,y\in I$\textit{\ and }$t\in \left[ 0,1\right] ,$\textit{\ we have}%
\begin{equation}
f\left( tx+\left( 1-t\right) y\right) \leq f\left( x\right) +f\left(
y\right) .  \label{104}
\end{equation}
\end{definition}

\begin{definition}
\textit{[\ref{hud}] Let }$s\in \left( 0,1\right] .$\textit{\ A function }$%
f:\left( 0,\infty \right] \rightarrow \left[ 0,\infty \right] $\textit{\ is
said to be }$s-$\textit{convex in the second sense if \ \ \ \ \ \ \ \ \ \ \
\ }%
\begin{equation}
f\left( tx+\left( 1-t\right) y\right) \leq t^{s}f\left( x\right) +\left(
1-t\right) ^{s}f\left( y\right) ,  \label{105}
\end{equation}%
\textit{for all }$x,y\in \left( 0,b\right] $\textit{\ \ and }$t\in \left[ 0,1%
\right] $\textit{. This class of }$\mathit{s-}$\textit{convex functions is
usually denoted by }$K_{s}^{2}$\textit{.}
\end{definition}

In 1978, Breckner introduced $s-$convex functions as a generalization of
convex functions in [\textit{\ref{bre1}}]. Also, in that work Breckner
proved the important fact that the set valued map is $s-$convex only if the
associated support function is $s-$convex function in [\textit{\ref{bre2}}].
A number of properties and connections with s-convex in the first sense are
discussed in paper [\textit{\ref{hud}}]. Of course, $s-$convexity means just
convexity when $s=1$.

\begin{definition}
\textit{[\ref{var}] Let }$h:J\subseteq 
\mathbb{R}
\rightarrow 
\mathbb{R}
$\textit{\ be a positive function . We say that }$f:I\subseteq 
\mathbb{R}
\rightarrow 
\mathbb{R}
$\textit{\ is }$h-$\textit{convex function, or that }$f$\textit{\ belongs to
the class }$SX\left( h,I\right) $\textit{, if }$f$\textit{\ is nonnegative
and for all }$x,y\in I$\textit{\ and }$t\in \left[ 0,1\right] $\textit{\ we
have \ \ \ \ \ \ \ \ \ \ \ \ }%
\begin{equation}
f\left( tx+\left( 1-t\right) y\right) \leq h\left( t\right) f\left( x\right)
+h\left( 1-t\right) f\left( y\right) .  \label{106}
\end{equation}
\end{definition}

If inequality (\ref{106}) is reversed, then $f$ is said to be $h-$concave,
i.e. $f\in SV\left( h,I\right) $. Obviously, if $h\left( t\right) =t$, then
all nonnegative convex functions belong to $SX\left( h,I\right) $\ and all
nonnegative concave functions belong to $SV\left( h,I\right) $; if $h\left(
t\right) =\frac{1}{t}$, then $SX\left( h,I\right) =Q\left( I\right) $; if $%
h\left( t\right) =1$, then $SX\left( h,I\right) \supseteq P\left( I\right) $%
; and if $h\left( t\right) =t^{s}$, where $s\in \left( 0,1\right) $, then $%
SX\left( h,I\right) \supseteq K_{s}^{2}$.

\begin{remark}
\lbrack \textit{\ref{var}}]Let $h$ be a non-negative function such that%
\begin{equation}
h\left( \alpha \right) \geq \alpha  \label{107}
\end{equation}%
for all $\alpha \in (0,1)$. For example, the function $h_{k}(x)=x^{k}$ where 
$k\leq 1$ and $x>0$ has that property. If $f$ is a non-negative convex
function on $I$ , then for $x,y\in I$ , $\alpha \in (0,1)$ we have 
\begin{equation}
f\left( \alpha x+(1-\alpha )y\right) \leq \alpha f(x)+(1-\alpha )f(y)\leq
h(\alpha )f(x)+h(1-\alpha )f(y).  \label{108}
\end{equation}%
So, $f\in SX(h,I)$. Similarly, if the function $h$ has the property: $%
h(\alpha )\leq \alpha $ for all $\alpha \in (0,1)$, then any non-negative
concave function $f$ belongs to the class $SV(h,I)$.
\end{remark}

For recent results and generalizations concerning $h-$convex functions see [%
\ref{bom}, \ref{bura}, \ref{oz}-\ref{var}] and references therein.

In [\textit{\ref{dr2}}], the following theorem which was obtained by
Dragomir and Agarwal contains the Hermite-Hadamard type integral inequality.

\begin{theorem}
\lbrack \textit{\ref{dr2}}] Let $f:I^{\circ }\subseteq 
\mathbb{R}
\rightarrow 
\mathbb{R}
$ be a differentiable mapping on $I^{\circ }$, $a,b\in I^{\circ }$ with $a<b$%
. If $|f^{\prime }|$ is convex on $[a,b]$, then the following inequality
holds:%
\begin{equation}
\left\vert \frac{f\left( a\right) +f\left( b\right) }{2}-\frac{1}{b-a}%
\int_{a}^{b}f\left( u\right) du\right\vert \leq \frac{\left( b-a\right)
\left( \left\vert f^{\prime }\left( a\right) \right\vert +\left\vert
f^{\prime }\left( b\right) \right\vert \right) }{8}.  \label{109}
\end{equation}
\end{theorem}

In [$\ref{kav}$] Kavurmaci et. al. used the following lemma and introduced
some new Hermite-Hadamard type inequalities:

\begin{lemma}
\label{lm1}$\left[ \ref{kav}\right] $Let $f:I\subseteq 
\mathbb{R}
\rightarrow 
\mathbb{R}
$ be a differentiable mapping on $I%
{{}^\circ}%
$, where $a,b\in I$ with $a<b$. If $f^{\prime }\in L\left[ a,b\right] $,
then the following equality holds:%
\begin{eqnarray}
&&\frac{\left( b-x\right) f\left( b\right) +\left( x-a\right) f\left(
a\right) }{b-a}-\frac{1}{b-a}\int_{a}^{b}f\left( u\right) du  \label{110} \\
&=&\frac{\left( x-a\right) ^{2}}{b-a}\int_{0}^{1}\left( t-1\right) f^{\prime
}\left( tx+\left( 1-t\right) a\right) dt  \notag \\
&&+\frac{\left( b-x\right) ^{2}}{b-a}\int_{0}^{1}\left( 1-t\right) f^{\prime
}\left( tx+\left( 1-t\right) b\right) dt.  \notag
\end{eqnarray}
\end{lemma}

\begin{theorem}
$\left[ \ref{kav}\right] $ Let $f:I\subseteq 
\mathbb{R}
\rightarrow 
\mathbb{R}
$ be a differentiable mapping on $I%
{{}^\circ}%
$ sucht hat $f^{\prime }\in L\left[ a,b\right] ,$ where $a,b\in I$ with $%
a<b. $ If $\left\vert f^{\prime }\right\vert ^{\frac{p}{p-1}}$ is convex on $%
\left[ a,b\right] $ and for some fixed $q>1,$\ then the following inequality
holds:%
\begin{eqnarray}
&&\left\vert \frac{\left( b-x\right) f\left( b\right) +\left( x-a\right)
f\left( a\right) }{b-a}-\frac{1}{b-a}\int_{a}^{b}f\left( u\right)
du\right\vert  \label{111} \\
&\leq &\left( \frac{1}{1+p}\right) ^{\frac{1}{p}}\left( \frac{1}{2}\right) ^{%
\frac{1}{p}}  \notag \\
&&\times \left[ \frac{\left( x-a\right) ^{2}\left[ \left\vert f^{\prime
}\left( a\right) \right\vert ^{q}+\left\vert f^{\prime }\left( x\right)
\right\vert ^{q}\right] ^{\frac{1}{q}}+\left( b-x\right) ^{2}\left[
\left\vert f^{\prime }\left( x\right) \right\vert ^{q}+\left\vert f^{\prime
}\left( b\right) \right\vert ^{q}\right] ^{\frac{1}{q}}}{b-a}\right]  \notag
\end{eqnarray}%
for each $x\in \left[ a,b\right] $ and $q=\frac{p}{p-1}.$
\end{theorem}

In [\ref{bar}], Barani et. al. proved a variant of Hadamard's inequality
which holds for $s-$convex functions in the second sense.

\begin{theorem}
$\left[ \ref{bar}\right] $ Let $f:I\subseteq 
\mathbb{R}
_{+}\rightarrow 
\mathbb{R}
$ be a differentiable mapping on $I%
{{}^\circ}%
$ and $a,b\in I%
{{}^\circ}%
$ with $a<b.$ If $\left\vert f^{\prime }\right\vert ^{\frac{p}{p-1}}$ is $s-$%
convex on $\left[ a,b\right] $ in the second sense, for some fixed $s\in
\left( 0,1\right] .$\ Then the following inequality holds,%
\begin{eqnarray}
&&\left\vert \frac{\left( b-x\right) f\left( b\right) +\left( x-a\right)
f\left( a\right) }{b-a}-\frac{1}{b-a}\int_{a}^{b}f\left( u\right)
du\right\vert  \label{112} \\
&\leq &\left( \frac{1}{1+p}\right) ^{\frac{1}{p}}\left( \frac{1}{s+1}\right)
^{\frac{1}{q}}\left\{ \frac{\left( x-a\right) ^{2}}{b-a}\left( \left\vert
f^{\prime }\left( x\right) \right\vert ^{q}+\left[ \left\vert f^{\prime
}\left( a\right) \right\vert ^{q}\right] ^{\frac{1}{q}}\right) \right. 
\notag \\
&&+\left. \frac{\left( b-x\right) ^{2}}{b-a}\left( \left\vert f^{\prime
}\left( x\right) \right\vert ^{q}+\left[ \left\vert f^{\prime }\left(
b\right) \right\vert ^{q}\right] ^{\frac{1}{q}}\right) \right\} ,  \notag
\end{eqnarray}%
where $q=\frac{p}{p-1}$.
\end{theorem}

The main purpose of this paper is to establish refinements inequalities of
the results in [\ref{kav}]. We obtained new inequalities related to the
right-hand side of Hermite-Hadamard inequality for functions when a power of
whose first derivatives in absolute values is $h-$convex. Then, we give some
applications for special means of real numbers.

\section{Main Results}

In this section we introduce some Hermite-Hadamard type inequalities for
h-convex functions with corollaries and remarks.

\begin{theorem}
\label{th1} Let $f:I\subseteq 
\mathbb{R}
\rightarrow 
\mathbb{R}
$ be a differentiable mapping on $I%
{{}^\circ}%
,$ where $a,b\in I$ with $a<b$ such that $f^{\prime }\in L\left[ a,b\right] $
and $h$ is supermultiplicative and nonnegative such that $h\left( \alpha
\right) \geq \alpha .$ If $\left\vert f^{\prime }\right\vert $ is $h-$convex
on $\left[ a,b\right] ,$\ then%
\begin{eqnarray}
&&\left\vert \frac{\left( b-x\right) f\left( b\right) +\left( x-a\right)
f\left( a\right) }{b-a}-\frac{1}{b-a}\int_{a}^{b}f\left( u\right)
du\right\vert  \label{21} \\
&\leq &\frac{\left( x-a\right) ^{2}}{b-a}\left( \left\vert f^{\prime }\left(
x\right) \right\vert \int_{0}^{1}h\left( \left( 1-t\right) t\right)
dt+\left\vert f^{\prime }\left( a\right) \right\vert \int_{0}^{1}h\left(
\left( 1-t\right) ^{2}\right) dt\right)  \notag \\
&&+\frac{\left( b-x\right) ^{2}}{b-a}\left( \left\vert f^{\prime }\left(
x\right) \right\vert \int_{0}^{1}h\left( \left( 1-t\right) t\right)
dt+\left\vert f^{\prime }\left( b\right) \right\vert \int_{0}^{1}h\left(
\left( 1-t\right) ^{2}\right) dt\right)  \notag
\end{eqnarray}

\begin{proof}
From Lemma \ref{lm1}, we have%
\begin{eqnarray*}
&&\left\vert \frac{\left( b-x\right) f\left( b\right) +\left( x-a\right)
f\left( a\right) }{b-a}-\frac{1}{b-a}\int_{a}^{b}f\left( u\right)
du\right\vert \\
&=&\left\vert \frac{\left( x-a\right) ^{2}}{b-a}\int_{0}^{1}\left(
t-1\right) f^{\prime }\left( tx+\left( 1-t\right) a\right) dt\right. \\
&&\left. +\frac{\left( b-x\right) ^{2}}{b-a}\int_{0}^{1}\left( 1-t\right)
f^{\prime }\left( tx+\left( 1-t\right) b\right) dt\right\vert \\
&\leq &\frac{\left( x-a\right) ^{2}}{b-a}\int_{0}^{1}\left( 1-t\right)
\left\vert f^{\prime }\left( tx+\left( 1-t\right) a\right) \right\vert dt \\
&&+\frac{\left( b-x\right) ^{2}}{b-a}\int_{0}^{1}\left( 1-t\right)
\left\vert f^{\prime }\left( tx+\left( 1-t\right) b\right) \right\vert dt
\end{eqnarray*}%
Since $\left\vert f^{\prime }\right\vert $ is $h-$convex, then we obtain%
\begin{eqnarray*}
&&\left\vert \frac{\left( b-x\right) f\left( b\right) +\left( x-a\right)
f\left( a\right) }{b-a}-\frac{1}{b-a}\int_{a}^{b}f\left( u\right)
du\right\vert \\
&\leq &\frac{\left( x-a\right) ^{2}}{b-a}\int_{0}^{1}\left( 1-t\right)
\left( h\left( t\right) \left\vert f^{\prime }\left( x\right) \right\vert
+h\left( 1-t\right) \left\vert f^{\prime }\left( a\right) \right\vert
\right) dt \\
&&+\frac{\left( b-x\right) ^{2}}{b-a}\int_{0}^{1}\left( 1-t\right) \left(
h\left( t\right) \left\vert f^{\prime }\left( x\right) \right\vert +h\left(
1-t\right) \left\vert f^{\prime }\left( b\right) \right\vert \right) dt \\
&\leq &\frac{\left( x-a\right) ^{2}}{b-a}\left( \left\vert f^{\prime }\left(
x\right) \right\vert \int_{0}^{1}h\left( 1-t\right) h\left( t\right)
dt+\left\vert f^{\prime }\left( a\right) \right\vert \int_{0}^{1}h^{2}\left(
1-t\right) dt\right) \\
&&+\frac{\left( b-x\right) ^{2}}{b-a}\left( \left\vert f^{\prime }\left(
x\right) \right\vert \int_{0}^{1}h\left( 1-t\right) h\left( t\right)
dt+\left\vert f^{\prime }\left( b\right) \right\vert \int_{0}^{1}h^{2}\left(
1-t\right) dt\right) \\
&\leq &\frac{\left( x-a\right) ^{2}}{b-a}\left( \left\vert f^{\prime }\left(
x\right) \right\vert \int_{0}^{1}h\left( \left( 1-t\right) t\right)
dt+\left\vert f^{\prime }\left( a\right) \right\vert \int_{0}^{1}h\left(
\left( 1-t\right) ^{2}\right) dt\right) \\
&&+\frac{\left( b-x\right) ^{2}}{b-a}\left( \left\vert f^{\prime }\left(
x\right) \right\vert \int_{0}^{1}h\left( \left( 1-t\right) t\right)
dt+\left\vert f^{\prime }\left( b\right) \right\vert \int_{0}^{1}h\left(
\left( 1-t\right) ^{2}\right) dt\right)
\end{eqnarray*}%
which completes the proof.
\end{proof}
\end{theorem}

\begin{corollary}
\label{c1}In Theorem \ref{th1}, if we choose $x=\frac{a+b}{2},$ we get%
\begin{eqnarray*}
&&\left\vert \frac{f\left( a\right) +f\left( b\right) }{2}-\frac{1}{b-a}%
\int_{a}^{b}f\left( u\right) du\right\vert \\
&\leq &\frac{b-a}{4}\left( 2\left\vert f^{\prime }\left( \frac{a+b}{2}%
\right) \right\vert \int_{0}^{1}h\left( \left( 1-t\right) t\right) dt+\left(
\left\vert f^{\prime }\left( a\right) \right\vert +\left\vert f^{\prime
}\left( b\right) \right\vert \right) \int_{0}^{1}h\left( \left( 1-t\right)
^{2}\right) dt\right)
\end{eqnarray*}
\end{corollary}

\begin{corollary}
\label{c2}In Corollary \ref{c1}, using the $h-$convexity of$\ |f^{\prime }|,$
we have%
\begin{eqnarray*}
&&\left\vert \frac{f\left( a\right) +f\left( b\right) }{2}-\frac{1}{b-a}%
\int_{a}^{b}f\left( u\right) du\right\vert \\
&\leq &\frac{b-a}{4}\left( \left\vert f^{\prime }\left( a\right) \right\vert
+\left\vert f^{\prime }\left( b\right) \right\vert \right) \left( 2h\left( 
\frac{1}{2}\right) \int_{0}^{1}h\left( \left( 1-t\right) t\right)
dt+\int_{0}^{1}h\left( \left( 1-t\right) ^{2}\right) dt\right) .
\end{eqnarray*}
\end{corollary}

\begin{corollary}
\label{c3}In Corollary \ref{c2}, if we choose $h\left( t\right) =t,$ we have%
\begin{eqnarray*}
&&\left\vert \frac{f\left( a\right) +f\left( b\right) }{2}-\frac{1}{b-a}%
\int_{a}^{b}f\left( u\right) du\right\vert \\
&\leq &\frac{b-a}{4}\left( \left\vert f^{\prime }\left( a\right) \right\vert
+\left\vert f^{\prime }\left( b\right) \right\vert \right) \left(
\int_{0}^{1}\left( 1-t\right) tdt+\int_{0}^{1}\left( 1-t\right) ^{2}dt\right)
\\
&=&\frac{b-a}{8}\left( \left\vert f^{\prime }\left( a\right) \right\vert
+\left\vert f^{\prime }\left( b\right) \right\vert \right)
\end{eqnarray*}

which is the inequality in (\ref{109}).
\end{corollary}

\begin{theorem}
\label{th2}Let $f:I\subseteq 
\mathbb{R}
\rightarrow 
\mathbb{R}
$ be a differentiable mapping on $I%
{{}^\circ}%
,$ where $a,b\in I$ with $a<b,$ such that $f^{\prime }\in L\left[ a,b\right] 
$ and $h:I\subseteq 
\mathbb{R}
\rightarrow 
\mathbb{R}
$ be a nonnegative such that $h\in L\left[ 0,1\right] .$ If $\left\vert
f^{\prime }\right\vert ^{\frac{p}{p-1}}$ is $h-$convex on $\left[ a,b\right] 
$ and for some fixed $q>1,$\ then%
\begin{eqnarray}
&&\left\vert \frac{\left( b-x\right) f\left( b\right) +\left( x-a\right)
f\left( a\right) }{b-a}-\frac{1}{b-a}\int_{a}^{b}f\left( u\right)
du\right\vert  \label{22} \\
&\leq &\left( \frac{1}{1+p}\right) ^{\frac{1}{p}}\left\{ \frac{\left(
x-a\right) ^{2}}{b-a}\left( \left\vert f^{\prime }\left( x\right)
\right\vert ^{q}\int_{0}^{1}h\left( t\right) dt+\left\vert f^{\prime }\left(
a\right) \right\vert ^{q}\int_{0}^{1}h\left( 1-t\right) dt\right) ^{\frac{1}{%
q}}\right.  \notag \\
&&+\left. \frac{\left( b-x\right) ^{2}}{b-a}\left( \left\vert f^{\prime
}\left( x\right) \right\vert ^{q}\int_{0}^{1}h\left( t\right) dt+\left\vert
f^{\prime }\left( b\right) \right\vert ^{q}\int_{0}^{1}h\left( 1-t\right)
dt\right) ^{\frac{1}{q}}\right\}  \notag
\end{eqnarray}%
for each $x\in \left[ a,b\right] $ and $q=\frac{p}{p-1}.$
\end{theorem}

\begin{proof}
From Lemma \ref{lm1} and using the well-known H\"{o}lder integral
inequality, we obtain

\begin{eqnarray*}
&&\left\vert \frac{\left( b-x\right) f\left( b\right) +\left( x-a\right)
f\left( a\right) }{b-a}-\frac{1}{b-a}\int_{a}^{b}f\left( u\right)
du\right\vert \\
&\leq &\frac{\left( x-a\right) ^{2}}{b-a}\int_{0}^{1}\left( 1-t\right)
\left\vert f^{\prime }\left( tx+\left( 1-t\right) a\right) \right\vert dt \\
&&+\frac{\left( b-x\right) ^{2}}{b-a}\int_{0}^{1}\left( 1-t\right)
\left\vert f^{\prime }\left( tx+\left( 1-t\right) b\right) \right\vert dt \\
&\leq &\left( \int_{0}^{1}\left( 1-t\right) ^{p}dt\right) ^{\frac{1}{p}%
}\left\{ \frac{\left( x-a\right) ^{2}}{b-a}\left( \int_{0}^{1}\left\vert
f^{\prime }\left( tx+\left( 1-t\right) a\right) \right\vert ^{q}dt\right) ^{%
\frac{1}{q}}\right. \\
&&+\left. \frac{\left( b-x\right) ^{2}}{b-a}\left( \int_{0}^{1}\left\vert
f^{\prime }\left( tx+\left( 1-t\right) b\right) \right\vert ^{q}dt\right) ^{%
\frac{1}{q}}\right\}
\end{eqnarray*}%
Hence, by $h-$convexity of $\left\vert f^{\prime }\right\vert ^{q},$ we have%
\begin{eqnarray*}
&&\left\vert \frac{\left( b-x\right) f\left( b\right) +\left( x-a\right)
f\left( a\right) }{b-a}-\frac{1}{b-a}\int_{a}^{b}f\left( u\right)
du\right\vert \\
&\leq &\left( \frac{1}{1+p}\right) ^{\frac{1}{p}}\left\{ \frac{\left(
x-a\right) ^{2}}{b-a}\left( \left\vert f^{\prime }\left( x\right)
\right\vert ^{q}\int_{0}^{1}h\left( t\right) dt+\left\vert f^{\prime }\left(
a\right) \right\vert ^{q}\int_{0}^{1}h\left( 1-t\right) dt\right) ^{\frac{1}{%
q}}\right. \\
&&+\left. \frac{\left( b-x\right) ^{2}}{b-a}\left( \left\vert f^{\prime
}\left( x\right) \right\vert ^{q}\int_{0}^{1}h\left( t\right) dt+\left\vert
f^{\prime }\left( b\right) \right\vert ^{q}\int_{0}^{1}h\left( 1-t\right)
dt\right) ^{\frac{1}{q}}\right\}
\end{eqnarray*}%
which completes the proof.
\end{proof}

\begin{corollary}
\label{c4}In Theorem \ref{th2}, if we choose $h\left( t\right) =t,$
inequality (\ref{22}) is reduces to \textbf{(}\ref{111}\textbf{)}
\end{corollary}

\begin{corollary}
\label{c5}In Theorem \ref{th2}, if we choose $h\left( t\right) =t^{s},$
inequality (\ref{22}) is reduces to (\ref{112})
\end{corollary}

\begin{corollary}
\label{c6}In Theorem \ref{th2}, if we choose $h\left( t\right) =1,$ then we
obtain an integral inequality for $P-$functions%
\begin{eqnarray*}
&&\left\vert \frac{\left( b-x\right) f\left( b\right) +\left( x-a\right)
f\left( a\right) }{b-a}-\frac{1}{b-a}\int_{a}^{b}f\left( u\right)
du\right\vert \\
&\leq &\left( \frac{1}{1+p}\right) ^{\frac{1}{p}}\left[ \frac{\left(
x-a\right) ^{2}+\left( b-x\right) ^{2}}{b-a}\left( \left\vert f^{\prime
}\left( x\right) \right\vert ^{q}+\left\vert f^{\prime }\left( a\right)
\right\vert ^{q}\right) ^{\frac{1}{q}}\right] .
\end{eqnarray*}
\end{corollary}

\begin{remark}
$\circ $In Theorem \ref{th2}, choosing $x=a,$ we get%
\begin{eqnarray*}
&&\left\vert f\left( b\right) -\frac{1}{b-a}\int_{a}^{b}f\left( u\right)
du\right\vert \\
&\leq &\left( b-a\right) \left( \frac{1}{1+p}\right) ^{\frac{1}{p}}\left(
\left\vert f^{\prime }\left( a\right) \right\vert ^{q}\int_{0}^{1}h\left(
t\right) dt+\left\vert f^{\prime }\left( b\right) \right\vert
^{q}\int_{0}^{1}h\left( 1-t\right) dt\right) ^{\frac{1}{q}}.
\end{eqnarray*}%
$\circ $In Theorem \ref{th2}, choosing $x=b,$ we get

\begin{eqnarray*}
&&\left\vert f\left( a\right) -\frac{1}{b-a}\int_{a}^{b}f\left( u\right)
du\right\vert \\
&\leq &\left( b-a\right) \left( \frac{1}{1+p}\right) ^{\frac{1}{p}}\left(
\left\vert f^{\prime }\left( b\right) \right\vert ^{q}\int_{0}^{1}h\left(
t\right) dt+\left\vert f^{\prime }\left( a\right) \right\vert
^{q}\int_{0}^{1}h\left( 1-t\right) dt\right) ^{\frac{1}{q}}.
\end{eqnarray*}%
$\circ $In Theorem \ref{th2}, choosing $x=\frac{a+b}{2},$ we get%
\begin{eqnarray*}
&&\left\vert \frac{f\left( b\right) +f\left( a\right) }{2}-\frac{1}{b-a}%
\int_{a}^{b}f\left( u\right) du\right\vert \\
&\leq &\frac{\left( b-a\right) }{4}\left( \frac{1}{1+p}\right) ^{\frac{1}{p}%
}\left\{ \left( \left\vert f^{\prime }\left( \frac{a+b}{2}\right)
\right\vert ^{q}\int_{0}^{1}h\left( t\right) dt+\left\vert f^{\prime }\left(
a\right) \right\vert ^{q}\int_{0}^{1}h\left( 1-t\right) dt\right) ^{\frac{1}{%
q}}\right. \\
&&+\left. \left( \left\vert f^{\prime }\left( \frac{a+b}{2}\right)
\right\vert ^{q}\int_{0}^{1}h\left( t\right) dt+\left\vert f^{\prime }\left(
b\right) \right\vert ^{q}\int_{0}^{1}h\left( 1-t\right) dt\right) ^{\frac{1}{%
q}}\right\} .
\end{eqnarray*}%
$\circ $In Theorem \ref{th2}, choosing $x=\frac{a+b}{2}$ and $f^{\prime
}\left( \frac{a+b}{2}\right) =0,$ we get%
\begin{eqnarray*}
&&\left\vert \frac{f\left( b\right) +f\left( a\right) }{2}-\frac{1}{b-a}%
\int_{a}^{b}f\left( u\right) du\right\vert \\
&\leq &\frac{\left( b-a\right) }{4}\left( \frac{1}{1+p}\right) ^{\frac{1}{p}}%
\left[ \left( \left\vert f^{\prime }\left( a\right) \right\vert +\left\vert
f^{\prime }\left( b\right) \right\vert \right) \left( \int_{0}^{1}h\left(
t\right) dt\right) ^{\frac{1}{q}}\right] .
\end{eqnarray*}
\end{remark}

\begin{theorem}
\label{th3}Let $f:I\subseteq 
\mathbb{R}
\rightarrow 
\mathbb{R}
$ be a differentiable mapping on $I%
{{}^\circ}%
,$ where $a,b\in I$ with $a<b,$ such that $f^{\prime }\in L\left[ a,b\right] 
$ and $h:I\subseteq 
\mathbb{R}
\rightarrow 
\mathbb{R}
$ be a nonnegative and supermultiplicative such that $h\in L\left[ 0,1\right]
$ and $h\left( t\right) \geq t.$ If $\left\vert f^{\prime }\right\vert ^{q}$
is $h-$convex on $\left[ a,b\right] $ and for some fixed $q>1,$\ then%
\begin{eqnarray*}
&&\left\vert \frac{\left( b-x\right) f\left( b\right) +\left( x-a\right)
f\left( a\right) }{b-a}-\frac{1}{b-a}\int_{a}^{b}f\left( u\right)
du\right\vert \\
&\leq &\left( \int_{0}^{1}h\left( 1-t\right) dt\right) ^{1-\frac{1}{q}} \\
&&\times \left\{ \frac{\left( x-a\right) ^{2}}{b-a}\left( \left\vert
f^{\prime }\left( x\right) \right\vert ^{q}\int_{0}^{1}h\left(
t-t^{2}\right) dt+\left\vert f^{\prime }\left( a\right) \right\vert
^{q}\int_{0}^{1}h\left( \left( 1-t\right) ^{2}\right) dt\right) ^{\frac{1}{q}%
}\right. \\
&&+\left. \frac{\left( b-x\right) ^{2}}{b-a}\left( \left\vert f^{\prime
}\left( x\right) \right\vert ^{q}\int_{0}^{1}h\left( t-t^{2}\right)
dt+\left\vert f^{\prime }\left( b\right) \right\vert ^{q}\int_{0}^{1}h\left(
\left( 1-t\right) ^{2}\right) dt\right) ^{\frac{1}{q}}\right\}
\end{eqnarray*}%
for each $x\in \left[ a,b\right] .$
\end{theorem}

\begin{proof}
From Lemma \ref{lm1} and using the power-mean inequality for $q\geq 1$ and $%
p=\frac{q}{q-1}$ we obtain%
\begin{eqnarray*}
&&\left\vert \frac{\left( b-x\right) f\left( b\right) +\left( x-a\right)
f\left( a\right) }{b-a}-\frac{1}{b-a}\int_{a}^{b}f\left( u\right)
du\right\vert \\
&\leq &\frac{\left( x-a\right) ^{2}}{b-a}\int_{0}^{1}\left( 1-t\right)
\left\vert f^{\prime }\left( tx+\left( 1-t\right) a\right) \right\vert dt \\
&&+\frac{\left( b-x\right) ^{2}}{b-a}\int_{0}^{1}\left( 1-t\right)
\left\vert f^{\prime }\left( tx+\left( 1-t\right) b\right) \right\vert dt \\
&\leq &\left( \int_{0}^{1}\left( 1-t\right) dt\right) ^{1-\frac{1}{q}%
}\left\{ \frac{\left( x-a\right) ^{2}}{b-a}\left( \int_{0}^{1}\left(
1-t\right) \left\vert f^{\prime }\left( tx+\left( 1-t\right) a\right)
\right\vert ^{q}dt\right) ^{\frac{1}{q}}\right. \\
&&+\left. \frac{\left( b-x\right) ^{2}}{b-a}\left( \int_{0}^{1}\left(
1-t\right) \left\vert f^{\prime }\left( tx+\left( 1-t\right) b\right)
\right\vert ^{q}dt\right) ^{\frac{1}{q}}\right\}
\end{eqnarray*}%
Hence, by $h-$convexity of $\left\vert f^{\prime }\right\vert ^{q},$ we get%
\begin{eqnarray*}
&&\left\vert \frac{\left( b-x\right) f\left( b\right) +\left( x-a\right)
f\left( a\right) }{b-a}-\frac{1}{b-a}\int_{a}^{b}f\left( u\right)
du\right\vert \\
&\leq &\left( \int_{0}^{1}\left( 1-t\right) dt\right) ^{1-\frac{1}{q}%
}\left\{ \frac{\left( x-a\right) ^{2}}{b-a}\left( \int_{0}^{1}\left(
1-t\right) \left\vert f^{\prime }\left( tx+\left( 1-t\right) a\right)
\right\vert ^{q}dt\right) ^{\frac{1}{q}}\right. \\
&&+\left. \frac{\left( b-x\right) ^{2}}{b-a}\left( \int_{0}^{1}\left(
1-t\right) \left\vert f^{\prime }\left( tx+\left( 1-t\right) b\right)
\right\vert ^{q}dt\right) ^{\frac{1}{q}}\right\} \\
&\leq &\left( \int_{0}^{1}h\left( 1-t\right) dt\right) ^{1-\frac{1}{q}%
}\left\{ \frac{\left( x-a\right) ^{2}}{b-a}\left( \int_{0}^{1}h\left(
1-t\right) \left\{ h\left( t\right) \left\vert f^{\prime }\left( x\right)
\right\vert ^{q}+h\left( 1-t\right) \left\vert f^{\prime }\left( a\right)
\right\vert ^{q}\right\} dt\right) ^{\frac{1}{q}}\right. \\
&&+\left. \frac{\left( b-x\right) ^{2}}{b-a}\left( \int_{0}^{1}h\left(
1-t\right) \left\{ h\left( t\right) \left\vert f^{\prime }\left( x\right)
\right\vert ^{q}+h\left( 1-t\right) \left\vert f^{\prime }\left( b\right)
\right\vert ^{q}\right\} dt\right) ^{\frac{1}{q}}\right\}
\end{eqnarray*}%
which completes the proof.
\end{proof}

\begin{corollary}
\label{c7}In Theorem \ref{th3}, choosing $x=\left( a+b\right) /2$ and then
using the $h-$convexity of $\left\vert f^{\prime }\right\vert ^{q},$ we get%
\begin{eqnarray*}
&&\left\vert \frac{f\left( a\right) +f\left( b\right) }{2}-\frac{1}{b-a}%
\int_{a}^{b}f\left( u\right) du\right\vert \\
&\leq &\frac{b-a}{4}\left( \int_{0}^{1}h\left( 1-t\right) dt\right) ^{\frac{1%
}{p}}\left\{ \left( \left\vert f^{\prime }\left( \frac{a+b}{2}\right)
\right\vert ^{q}\int_{0}^{1}h\left( t-t^{2}\right) dt+\left\vert f^{\prime
}\left( a\right) \right\vert ^{q}\int_{0}^{1}h\left( \left( 1-t\right)
^{2}\right) dt\right) ^{\frac{1}{q}}\right. \\
&&+\left. \left( \left\vert f^{\prime }\left( \frac{a+b}{2}\right)
\right\vert ^{q}\int_{0}^{1}h\left( t-t^{2}\right) dt+\left\vert f^{\prime
}\left( b\right) \right\vert ^{q}\int_{0}^{1}h\left( \left( 1-t\right)
^{2}\right) dt\right) ^{\frac{1}{q}}\right\} .
\end{eqnarray*}
\end{corollary}

\begin{remark}
\bigskip \bigskip In Corollary \ref{c7}, choosing $h\left( t\right) =t,$ we
get%
\begin{eqnarray*}
&&\left\vert \frac{f\left( a\right) +f\left( b\right) }{2}-\frac{1}{b-a}%
\int_{a}^{b}f\left( u\right) du\right\vert \\
&\leq &\frac{b-a}{4}\left( \frac{1}{2}\right) ^{\frac{1}{p}}\left( \frac{1}{3%
}\right) ^{\frac{1}{q}}\left\{ \left( \frac{1}{2}\left\vert f^{\prime
}\left( \frac{a+b}{2}\right) \right\vert ^{q}+\left\vert f^{\prime }\left(
a\right) \right\vert ^{q}\right) ^{\frac{1}{q}}\right. \\
&&+\left. \left( \frac{1}{2}\left\vert f^{\prime }\left( \frac{a+b}{2}%
\right) \right\vert ^{q}+\left\vert f^{\prime }\left( b\right) \right\vert
^{q}\right) ^{\frac{1}{q}}\right\}
\end{eqnarray*}
\end{remark}

\begin{corollary}
\label{c8}In Corollary \ref{c7}, if we choose $h\left( t\right) =1,$ then we
obtain an integral inequality for $P-$functions%
\begin{eqnarray*}
&&\left\vert \frac{f\left( a\right) +f\left( b\right) }{2}-\frac{1}{b-a}%
\int_{a}^{b}f\left( u\right) du\right\vert \\
&\leq &\frac{b-a}{4}\left\{ \left( \left\vert f^{\prime }\left( \frac{a+b}{2}%
\right) \right\vert ^{q}+\left\vert f^{\prime }\left( a\right) \right\vert
^{q}\right) ^{\frac{1}{q}}\right. \\
&&+\left. \left( \left\vert f^{\prime }\left( \frac{a+b}{2}\right)
\right\vert ^{q}+\left\vert f^{\prime }\left( b\right) \right\vert
^{q}\right) ^{\frac{1}{q}}\right\} \\
&\leq &\frac{b-a}{2}\left( \left\vert f^{\prime }\left( a\right) \right\vert
+\left\vert f^{\prime }\left( b\right) \right\vert \right) .
\end{eqnarray*}
\end{corollary}

\section{Applications to Special Means}

We now consider the applications of our Theorems to the following special
means

The quadratic mean: 
\begin{equation*}
K=K\left( a,b\right) :=\sqrt{\frac{a^{2}+b^{2}}{2}}\ a,b\geq 0,
\end{equation*}

The arithmetic mean: 
\begin{equation*}
A=A\left( a,b\right) :=\frac{a+b}{2},\ \ a,b\geq 0,
\end{equation*}

The geometric mean: 
\begin{equation*}
G=G\left( a,b\right) :=\sqrt{ab},\ a,b\geq 0,
\end{equation*}

The logarithmic mean: 
\begin{equation*}
L=L\left( a,b\right) :=\left\{ 
\begin{array}{c}
a\text{ \ \ \ \ \ \ \ \ \ \ \ \ \ if \ \ }a=b \\ 
\frac{b-a}{\ln b-\ln a}\text{ \ \ \ \ \ if \ \ }a\neq b%
\end{array}%
\right. ,\ a,b\geq 0
\end{equation*}

The p-logarithmic mean:%
\begin{eqnarray*}
L_{p} &=&L_{p}\left( a,b\right) :=\left\{ 
\begin{array}{l}
\left[ \frac{b^{p+1}-a^{p+1}}{\left( p+1\right) \left( b-a\right) }\right]
^{1/p}\text{ \ \ \ \ \ if \ \ }a\neq b \\ 
a\text{ \ \ \ \ \ \ \ \ \ \ \ \ \ \ \ \ \ \ \ \ \ \ if \ \ }a=b%
\end{array}%
\right. , \\
\ p &\in &%
\mathbb{R}
\backslash \left\{ -1,0\right\} ;\ a,b>0.
\end{eqnarray*}

\begin{proposition}
\label{p1}Let $a,b\in 
\mathbb{R}
,$ $0<a<b$ and $n\in 
\mathbb{Z}
,$ $\left\vert n\right\vert \geq 1$. Then, the following inequality holds,%
\begin{equation}
\left\vert A\left( a^{n},b^{n}\right) -L_{n}^{n}\left( a,b\right)
\right\vert \leq \left\vert n\right\vert \frac{b-a}{12}\left( A^{n-1}\left(
a,b\right) +2A\left( a^{n-1},b^{n-1}\right) \right) .  \label{301}
\end{equation}
\end{proposition}

\begin{proof}
If we apply Corollary \ref{c1} for $f\left( x\right) =x^{n}$, $h\left(
t\right) =t$ where $x\in 
\mathbb{R}
,$ $n\in 
\mathbb{Z}
,$ $\left\vert n\right\vert \geq 1,$\ we get the proof (\ref{301}).
\end{proof}

\begin{proposition}
\label{p2}Let $a,b\in 
\mathbb{R}
,$ $0<a<b$ and $n\in 
\mathbb{Z}
,$ $\left\vert n\right\vert \geq 1$. Then, we have:%
\begin{equation}
\left\vert A\left( a^{-1},b^{-1}\right) -L^{-1}\left( a,b\right) \right\vert
\leq \frac{b-a}{12}\left( \frac{1}{A^{2}\left( a,b\right) }+\frac{%
2K^{2}\left( a,b\right) }{G^{4}\left( a,b\right) }\right)  \label{302}
\end{equation}
\end{proposition}

\begin{proof}
The proof is immediate from Corollary \ref{c1} applied for $f(x)=\frac{1}{x}$%
, $x\in \left[ a,b\right] $ and $h\left( t\right) =t.$
\end{proof}

\begin{proposition}
\label{p3}Let $a,b\in 
\mathbb{R}
,$ $0<a<b$ and $n\in 
\mathbb{Z}
,$ $\left\vert n\right\vert \geq 1$. Then, the following inequality holds,%
\begin{equation}
\left\vert A\left( a^{n},b^{n}\right) -L_{n}^{n}\left( a,b\right)
\right\vert \leq \frac{b-a}{4}\left\vert n\right\vert A\left(
a^{n-1},b^{n-1}\right) .  \label{303}
\end{equation}
\end{proposition}

\begin{proof}
The assertion follows from Corollary \ref{c2} applied to $f(x)=x^{n}$, $%
h\left( t\right) =t$ where $x\in 
\mathbb{R}
,$ $n\in 
\mathbb{Z}
,$ $\left\vert n\right\vert \geq 1$.
\end{proof}

\begin{proposition}
\label{p4}Let $a,b\in 
\mathbb{R}
,$ $0<a<b$ and $n\in 
\mathbb{Z}
,$ $\left\vert n\right\vert \geq 1$. Then, we have:%
\begin{equation}
\left\vert A\left( a^{-1},b^{-1}\right) -L^{-1}\left( a,b\right) \right\vert
\leq \frac{b-a}{8}\frac{2K^{2}\left( a,b\right) }{G^{4}\left( a,b\right) }
\label{304}
\end{equation}
\end{proposition}

\begin{proof}
The assertion follows from Corollary \ref{c2} applied to $f(x)=\frac{1}{x}$, 
$x\in \left[ a,b\right] $ and $h\left( t\right) =t$.
\end{proof}

\begin{proposition}
\label{p5}Let $a,b\in 
\mathbb{R}
,$ $0<a<b$ and $n\in 
\mathbb{Z}
,$ $\left\vert n\right\vert \geq 1$. Then, for all $q>1:$%
\begin{eqnarray}
&&\left\vert A\left( a^{n},b^{n}\right) -L_{n}^{n}\left( a,b\right)
\right\vert  \label{305} \\
&\leq &\left\vert n\right\vert \left( b-a\right) \left( 2\right) ^{\frac{1}{q%
}-3}\left( 3\right) ^{-\frac{1}{q}}\left\{ \left( \frac{A^{q\left(
n-1\right) }\left( a,b\right) }{2}+a^{q\left( n-1\right) }\right) ^{\frac{1}{%
q}}\right.  \notag \\
&&+\left. \left( \frac{A^{q\left( n-1\right) }\left( a,b\right) }{2}%
+b^{q\left( n-1\right) }\right) ^{\frac{1}{q}}\right\} .  \notag
\end{eqnarray}
\end{proposition}

\begin{proof}
If we apply Corollary \ref{c7} for $f(x)=x^{n}$ where $x\in 
\mathbb{R}
,$ $n\in 
\mathbb{Z}
,$ $\left\vert n\right\vert \geq 1,$\ we get the proof (\ref{305}).
\end{proof}

\end{document}